\documentclass[12pt,reqno,oneside]{amsart}
\usepackage{amsmath,amsthm,amsfonts,amssymb}
\usepackage[mathscr]{eucal}
\usepackage{indentfirst}
\usepackage{url}

\theoremstyle{plain}
\newtheorem{teo}{Theorem}[section]
\newtheorem{cor}[teo]{Corollary}
\newtheorem{lem}[teo]{Lemma}
\newtheorem{prop}[teo]{Proposition}

\theoremstyle{definition}

\newtheorem{rem}[teo]{Remark}

\numberwithin{equation}{section}

\def\bbR{{\mathbb R}}
\def\bbP{{\mathbb P}}

\def\bbN{{\mathbb N}}
\def\bbE{{\mathbf E}}

\begin{document}

\baselineskip=26pt

\title[A species survival time and the Bessel distributions]{On a link between a species survival \\ time in an evolution model and \\ the Bessel distributions}

\author{Herv\'e Guiol}
\author{F\'abio~P.~Machado}
\author{Rinaldo Schinazi}

\address[Herv\'e Guiol]
{UJF-Grenoble 1 / CNRS / Grenoble INP / TIMC-IMAG UMR 5525, Grenoble, F-38041, France}
\email{Herve.Guiol@imag.fr}

\address[F\'abio~P.~Machado]
{University of S\~ao Paulo, Brazil.}
\noindent
\email{fmachado@ime.usp.br}

\address[Rinaldo Schinazi]
{UCCS at Colorado Springs, USA}
\noindent
\email{rschinaz@uccs.edu}

\thanks{Research supported by CNPq (306927/2007-1), ANR Blanc SHEPI and PICS (5470).}



\date{\today}

\begin{abstract}
We consider a stochastic model for species evolution. A new species is born at rate $\lambda$ 
and a species dies at rate $\mu$. A random number, sampled from a given distribution $F$, is 
associated with each new species at the time of birth. Every time there is a death event, 
the species that is killed is the one with the smallest fitness. We consider the (random) 
survival time of a species with a given fitness $f$. We show that the survival time distribution 
depends crucially on whether $f<f_c$, $f=f_c$ or $f>f_c$ where $f_c$ is a critical fitness that 
is computed explicitly.
\end{abstract}

\maketitle

\section{Introduction}
\label{S: Introduction}

Consider a stochastic model for species evolution in which a new species is born at
rate $\lambda$ and an existing species dies at rate $\mu$. A random number, sampled
from a given distribution $F$, is associated with each new species at the
time of birth. We think of the random number associated with a given species as
being the \textit{fitness} of the species. These fitnesses are independent
of each other and of everything else in the process.
Every time there is a death event, the species that is killed is the one with the smallest
fitness.  We assume $F$ to be an absolute continuous distribution function.  In this paper we study
the survival time of a given species with fitness $f$. We show that there is a critical fitness $f_c$ and a sharp phase transition for the survival time of the species. Our analysis is based on a closed connection between our model and random walks.

A similar model, meant to build phylogenetic trees, was introduced in Liggett and Schinazi~\cite{LS}.  A discrete version of this model is studied in Guiol \textit{et al}~\cite{GMS} where a phase transition is shown.

\section{Main Results}\label{S: MR}

Recall we assume that $F$ is an absolute continuous distribution function. This implies that
there exists a probability density $\varphi$ on $\bbR$ such that
\[
F(x)=\int_{-\infty}^x\varphi(u)\ du.
\]
Denote $\mathrm{Supp}(F)$ the support of distribution $F$:
\[
\mathrm{Supp}(F)=\{x\in\bbR : \varphi(x)>0 \}
\]
Assume that originally there are $k$ species in the process, with associate fitness $f_1<f_2<...<f_k=f$
in the support of $F$. Denote $\tau^k_f$ the r.v corresponding to the survival time of the species with
fitness $f$ in that context.
\begin{teo}\label{T: ClosedForm}
The survival time $\tau^k_f$ has a Bessel distribution :
\begin{equation}\label{E: Survival}
\bbP(\tau^k_f>t)=1-\left(\frac{\mu}{\lambda_f}\right)^{k/2}\int_0^te^{-(\mu+\lambda_f)u}{\frac{k}{u}}I_k(2\sqrt{\mu\lambda_f}\ u)\ du
\end{equation}
with $\lambda_f:=\lambda F(f)<\lambda$ and
where $I_k$ is the modified Bessel function of the first kind with indice $k$ defined by
\begin{equation}\label{E: Besselk}
I_k(x)=\sum_{\ell=0}^{+\infty}\frac{1}{(\ell+k)!\ell!}\left(\frac{x}{2}\right)^{2\ell+k}.
\end{equation}
\end{teo}

\begin{rem}
Devroye~\cite[chapter IX section 7 p.470]{Dev} includes distribution \eqref{E: Survival} among the class of Bessel function distributions. One can also find it in Feller~\cite[chapter II section 9 Problem 15 p.65]{Fel} without a specific name.
\end{rem}
\begin{rem}
The survival time $\tau^k_f$ is not affected by living species with fitness above $f$.
\end{rem}
\begin{rem}\label{R : transient}
Whenever $\lambda_f\leq \mu$, $\tau^k_f$ has probability density 
\[
\varphi_{\tau}(t)=\left(\frac{\mu}{\lambda_f}\right)^{k/2}e^{-c t}{\frac{k}{t}}I_k(2\sqrt{\mu\lambda_f}\ t)\mbox{ for }t>0.
\]
In the case $\lambda_f>\mu$ the previous function is not a density probability since
\[
\int_0^{\infty}\left(\frac{\mu}{\lambda_f}\right)^{k/2}e^{-c t}{\frac{k}{t}}I_k(2\sqrt{\mu\lambda_f}\ t)\ dt
=\left(\frac{\mu}{\lambda_f}\right)^{k/2}\left(\frac{\mu}{\lambda_f}\right)^{k/2}=\left(\frac{\mu}{\lambda_f}\right)^k<1
\]
\end{rem}

Though formula \eqref{E: Survival} gives the exact distribution of the survival time, it is not straightforward to
come to a conclusion from it. Next result sheds light in the phase transition property of our model.

\begin{cor}
\label{C: UpperBounds}
With the previous notation
\begin{enumerate}
\item[\textit{(a)}] If $\lambda_f<\mu$ then
\[
\bbP(\tau^k_f>t) \sim C_k\ \frac{e^{-\gamma t}}{t^{3/2}} \]
with $C_k= \frac{k}{2\sqrt{\pi}}\left(\frac{\mu}{\lambda_f}\right)^{k/2}
(\mu\lambda_f)^{-1/4}(\sqrt{\mu}-\sqrt{\lambda_f})^{-2}$
and $\gamma=(\sqrt{\mu}-\sqrt{\lambda_f})^2.$
\item[\textit{(b)}] If $\lambda_f>\mu$ then
\begin{eqnarray*}
\bbP(\tau^k_f=+\infty)&=&1-\left({\frac{\mu}{\lambda_f}}\right)^k;\\
\bbP(+\infty>\tau^k_f>t) &\sim& C_k\ \frac{e^{-\gamma t}}{t^{3/2}}.
\end{eqnarray*}
\item[\textit{(c)}] If $\lambda_f=\mu$ then
\[ \bbP(\tau^k_f>t) \sim k(\pi \mu t)^{-{1/2}}.\]
\end{enumerate}
\end{cor}

\begin{rem} Note, from Corollary~\ref{C: UpperBounds},
that if $\lambda>\mu$ there is a phase transition in $f$.
A species born with a fitness lower than
\begin{equation}
f_c:=F^{-1}(\mu/\lambda)
\end{equation}
dies out exponentially fast while a species with a fitness greater than $f_c$
has a positive probability of surviving forever.
The larger $\lambda/\mu$ (recall that $F^{-1}$ is non-decreasing) the more welcoming the environment is to
new species. If $\lambda/\mu<1$  all species will die exponentially fast.
On the other hand if $\lambda/\mu$ is large then even species with relatively
low fitness will make it.
\end{rem}

Up to now we have discussed the survival of a species with a given fitness $f$. It is particularly
relevant to derive some information about the distribution of these surviving species.
Suppose that $\lambda>\mu$ and let $L_t$ and $R_t$ be the sets
of species alive at time $t$ whose fitness is respectively lower than $f_c$ and higher than $f_c$.
Since each fitness that has appeared up to time $t$ will not show up again a.s., we can identify
each species with its fitness and think of $L_t$ and $R_t$ as sets of points in $(-\infty,f_c)$ and $(f_c,\infty)$ respectively.
Next result states a straightforward application of the main result of Guiol \textit{et al}~\cite{GMS}

\begin{prop}
\label{P: SameDistribution}
Suppose that $\lambda>\mu$.  Then

(a)  The number $|L_t|$ of species whose fitness is below $f_c$ is a null recurrent birth and death process. In particular, the set $L_t$ is empty infinitely often with probability one.

(b) Let $f_c<a<b$ then
 \[
 \lim_{t\to\infty}{\frac{1}{t}} |R_t \cap (a,b)| = \frac{\lambda (F(b)-F(a))}{\lambda + \mu} \hbox { a.s.}
 \]
\end{prop}

\begin{proof}[Proof of Proposition~\ref{P: SameDistribution}]
To see this observe that the embedded discrete Markov chain for our process is the \textit{stochastic model of evolution}
defined in Guiol \textit{et al}~\cite{GMS} in such a way that $p = {\lambda}/{(\lambda + \mu)}$. Consider that
whenever the total number of species is 0, the death marks (see page~\pageref{deathmarks} below)
in the construction of the process are ignored so the total
number of species stays 0 with probability ${\mu}/{(\lambda + \mu)}$.
\end{proof}

\begin{rem}
Observe that from Ben Ari \emph{et al.}~\cite{BMR} it would be also possible to get a Central Limit Theorem and a Law of the Iterate logarithm for $R_t$.
\end{rem}

\section{Proofs}
\label{S: proofs}

\subsection{Construction of the process.}

The construction uses ideas from Harris Graphical method for Markov Processes and basically take advantage from
projections properties of a bi-dimensional Poisson process with rate 1. In the sequel we construct a bi-variate process $Z_t=(Z^1_t,Z^2_t)$ in which $Z^1_t$ will represent the number of living species at time $t$, and $Z^2_t$ will be the set of associated living fitness: In particular $|Z^2_t|=Z^1_t$, where $|A|$ denotes the cardinal of set $A$.

Let $M$ be a two dimensional Poisson process with rate 1 on $\bbR^+\times\bbR$. For notational convenience we will
identify the $x$-line of the plane as the time line.

Suppose we start the process with $k\geq 1$ species, let $f_1,...,f_k$ be $k$ independent random variables with $F$ distribution, independent from $M$. 

Let $T_0=0$ and $Z_0=(Z_0^1,Z_0^2)=(k,\{f_1,...,f_k\})\in \bbN\times \mathbf{S}$ where $\mathbf{S}$
is the set of finite subsets of real numbers in $[0,1]^{\bbN}$.

Define
\begin{equation}\label{E : infimum}
T_1=\inf\{t>0:M([0,t]\times[0,\lambda+\mu])>0\}
\end{equation}
that is the first time $t\in \bbR^+$ that a Poisson mark falls into the strip $\bbR^+\times[0,\lambda+\mu]$.
Denote by $(T_1,Y_1)$ the coordinate of the Poisson mark realizing the infimum in \eqref{E : infimum}.
Observe that from the Poisson process properties $Y_1$ is a uniform $[0,1]$ r.v. independent of $T_1$.
\begin{itemize}
\item Whenever $Y_1\in[0,\lambda]$ then let $f_{k+1}=F(Y_1/\lambda)$ (observe that $f_{k+1}$ is also a r.v. with $F$ law and independent of $T_1$) and let
\[
Z_{T_1}=(Z_0^1+1,Z_0^2\cup\{f_{k+1}\})=(k+1,\{f_0,...,f_{k+1}\})
\]
this will represent the birth of a new species;
\item else, whenever $Y_1\in]\lambda,\lambda+\mu]$, let
    \[
    Z_{T_1}=(Z_0^1,Z_0^2\setminus\min\{Z_0^2\})=(k-1,\{f_0,...,f_{k}\}\setminus\min\{f_i:1\leq i\leq k\})
    \]
    this will represent the death{\label{deathmarks}} of the weakest species.
\end{itemize}

For all $t\in[0,T_1[$ denote $Z_t=Z_0$. We have thus construct the process $Z_t$ until time $T_1$ (included).\\
For all $n\geq 1$ denote by $T_n$ the time of the $n$-th mark of the Poisson process $M$ in the strip $\bbR^+\times[0,\lambda+\mu]$ that is
\[
T_{n}=\inf\{t>T_{n-1}:M([0,t]\times[0,\lambda+\mu])>0\}.
\]
Suppose the process $Z_t$ is construct up to time $T_n$, $n\geq 1$.
As before denote by $(T_{n+1},Y_{n+1})$ the coordinate of the $n+1$st Poisson mark.
\begin{itemize}
\item Whenever $Y_{n+1}\in[0,\lambda]$ then let $f_{k+n+1}=F(Y_{n+1}/\lambda)$ and define
\[
Z_{T_{n+1}}=(Z_{T_n}^1+1,Z_{T_n}^2\cup\{f_{k+n+1}\})
\]
\item else
\[
Z_{T_{n+1}}=(Z_{T_n}^1-{\bf{1}}_{\{Z_{T_n}^1>0\}},Z_{T_n}^2\setminus\min\{Z_{T_n}^2\}).
\]
with the convention $\min\emptyset=\emptyset$;
\end{itemize}
then for all $t\in[T_n,T_{n+1}[$ let $Z_t=Z_{T_{n}}$.

So by induction one can construct the process $(Z_t)_{t\geq 0}$
so that the second coordinate of $Z_t$ i.e. $Z_t^2$ represents our fitness process starting with $k$
species.

\subsection{An useful coupling}

From the previous construction one can couple $Z_t=(Z^1_t,Z^2_t)$ with another process $X_t\in\bbN$ as follows.

Let $X_0=Z^1_0$ (with the construction's notation) and denote by $f=\max\{f_1,...,f_k\}$. At time $T_1$
\begin{itemize}
\item if $Y_1\in[0,\lambda_f]\cup ]\lambda,\lambda+\mu]$
then let $X_{T_1}=Z^1_{T_1}$. Observe that this corresponds to a simultaneous death or to a simultaneous birth  with associate fitness less than $f$ for the $Z$ process;
\item else (when $Y_1\in]\lambda_f,\lambda]$) then let $X_{T_1}=X_{0}$. In this case there is a birth on the $Z$ process with associate fitness bigger than $f$ and nothing for the $X$ process.
\end{itemize}
As before define $X_t=X_0$ for all $t\in[0,T_1[$.

For all set $A$ of numbers in $[0,1]$ denote by
\begin{equation}
\phi_f(A)=\{x\in A:x\leq f\},
\end{equation}
i.e. the set of numbers in $A$ less or equal to $f$.

Observe that $X_t=Z^1_t$ on $[0,T_1[$, $X_{T_1}=|\phi_f(Z^2_{T_1})|\leq Z^1_{T_1}$.

For $n\geq 1$ suppose that $X_t$ is construct up to time $T_n$.\\
If $X_{T_n}\neq 0$ then $|\phi_f(Z^2_{T_{n}})|=X_{T_n}$
\begin{itemize}
\item if $Y_{n+1}\in[0,\lambda_f]$ (recall that $f_{k+n+1}:=F(Y_{n+1}/\lambda)\leq f$) define
\[
X_{T_{n+1}}=X_{T_{n}}^1+1,
\]
\item if $Y_{n+1}\in[\lambda,\lambda+\mu]$ then define
    \[
    X_{T_{n+1}}=X_{T_{n}}^1-1
    \]
\item  else let $X_{T_{n+1}}=X_{T_{n}}$;
\end{itemize}
In case $X_{T_n}=0$ let $X_{T_{n+1}}=0$.

This defines a random sequence $(T_n,X_{T_n})_{n\geq 0}$, we define the process $(X_t)_{t\geq 0}$ as
$X_t=X_{T_{n}}$ for all $t\in[ T_{n},T_{n+1}[$.


The proof of Theorem \ref{T: ClosedForm} relies on the following Lemma.
\begin{lem}
\label{L: TaufXt}
For any $k\geq 1$
\begin{equation}\label{E : Coupling}
\{\tau^k_f>t\}=\{X_t>0\}
\end{equation}
i.e. $\tau^k_f$ has the same law as the first passage time to $0$
of $X_t$ the simple Bernoulli random walk starting at $k$ with rate $c=\lambda_f+\mu$
and individual steps equal to $1$ or $-1$ with respective probability $p=\lambda/c$ and $q=\mu/c$.
\end{lem}

\begin{proof}[Proof of Lemma~\ref{L: TaufXt}]
We have $X_0=Z^1_t=|\phi_f(Z^2_t)|=k>0$. From the construction for all $t<\tau^k_f$ we have $\min Z_{t}^2\leq f$ this implies $X_{t}>0$.
Conversely if $\tau^k_f\leq t$ as $\min Z_{\tau^k_f}^2> f$ this implies $X_{\tau^k_f}=0$ and thus
$X_t^1=0$.
\end{proof}

\begin{proof}[Proof of Theorem \ref{T: ClosedForm}]
Let $(T_n)_{n\geq 1}$ denotes the jump times of the process $(X_t)_{t\geq 0}$ and set $T_0=0$. 
The random sequence $(X_{T_n})_{n\geq 0}$ is a simple discrete time random walk on $\bbN$ with individual steps 
equal to $1$ or $-1$ with respective probability $p$ and $q$. Denote by $H_0$ the first hitting time of $0$
of this walk. A standard computation (see for instance Grimmett-Stirzaker~\cite[(15) p.79]{GS}) gives
\[
\bbP(H_0=n|X_0=k)=\frac{k}{n}{n \choose (n+k)/2}q^{(n+k)/2}p^{(n-k)/2}
\]
whenever $n+k$ is even, $0$ otherwise.
As $T_n$ has a Gamma distribution with parameters $c$ and $n$ this implies that
\begin{eqnarray*}
\bbP(X_t=0)&=&\int_0^{t}\sum_{n=k}^{\infty}\frac{c^{n}}{(n-1)!}u^{n-1}e^{-cu}
\bbP(H_0=n|X_0=k) \ du\\
&=&\int_0^{t}e^{-cu}\frac{k}{u}\sum_{n=k}^{\infty}\frac{1}{n!(n-k)!}(cu)^{2n-k}q^{n}p^{n-k} \ du\\
&=&\int_0^{t}e^{-cu}\frac{k}{u}\sum_{\ell=0}^{\infty}\frac{1}{(\ell+k)!\ell!}(cu)^{2\ell+k}q^{\ell+k}p^{\ell} \ du\\
&=&\int_0^{t}e^{-cu}\frac{k}{u}\left(\frac{q}{p}\right)^{k/2}\sum_{\ell=0}^{\infty}\frac{1}{(\ell+k)!\ell!}
(cu\sqrt{qp})^{2\ell+k} \ du
\end{eqnarray*}
and from the definition of the Bessel function \eqref{E: Besselk}
\[
\bbP(X_t=0)=\int_0^{t}e^{-cu}\frac{k}{u}\left(\frac{q}{p}\right)^{k/2}I_k(2cu\sqrt{pq})
\]
\end{proof}


\begin{rem}
Let $\tau^0_f=0$ We have
\[
\tau^k_f=\sum_{j=1}^{k}\tau^j_f-\tau^{j-1}_f.
\]
Observe that $\tau^k_f$ are a.s. finite stopping times and from the Strong Markov property
$(\tau^j_f-\tau^{j-1}_f)_{1\leq j\leq k}$ is an i.i.d. sequence of r.v. with the distribution of
$\tau_f:=\tau^1_f$.

As observed in Remark \ref{R : transient} $\lambda_f<\mu$ the expression \eqref{E: Survival} gives
the density probability of $\tau_f$:
\[
\varphi_{\tau_f}(t)=\sqrt{\frac{\mu}{\lambda_f}}e^{-(\mu+\lambda_f)t}\frac{1}{t}I_1(2\sqrt{\mu\lambda_f}t)
\]
for $t>0$.
Which allows to compute its Moment Generating Function:
\[
M(s)=\bbE(e^{-\tau_f s})=\frac{2\mu}{\sqrt{(s+\mu+\lambda_f)^2-4\mu\lambda_f}+s+\mu+\lambda_f}.
\]
This in turns allows us to compute $\bbE(\tau_f)=\frac{2\mu}{\mu-\lambda_f}$. So that one can see easily that
\[
\bbE(\tau^k_f)=k\ \frac{2\mu}{\mu-\lambda_f}
\]
for all $k\geq 0$.
\end{rem}


\begin{proof}[Proof of Corollary~\ref{C: UpperBounds} (a)]
 When $\lambda_f<\mu$ \eqref{E: Survival} reads
\[
\bbP(\tau_f>t)=\left(\frac{\mu}{\lambda_f}\right)^{k/2}\int_t^{+\infty}e^{-(\mu+\lambda_f)u}\frac{k}{u}I_k(2\sqrt{\mu\lambda_f}u)\ du
\]

Since (see Arfken and Weber, of~\cite[section 11.6 eq.11.137 p.719]{AW}) for all $k\geq 1$

\begin{equation}\label{E: Bessel1}
\frac{e^x}{\sqrt{2\pi x}}(1-\frac {4k^2-1}{8x})\leq I_k(x)\leq \frac{e^x}{\sqrt{2\pi x}}
\end{equation}

for $x$ large enough,

\begin{eqnarray}
&\displaystyle\frac 1{2\sqrt{\pi}(\mu\lambda_f)^{1/4}}\ \frac{e^{-(\sqrt{\mu}-\sqrt{\lambda_f})^2u}}{u^{3/2}}\left(1-\frac {4k^2-1}{16\sqrt{\mu\lambda_f}\ u}\right)&\nonumber\\
&\displaystyle\leq \frac{e^{-(\mu+\lambda_f)u}}{u}I_1(2\sqrt{\mu\lambda_f}\ u)\leq&\nonumber\\
&\displaystyle
\frac 1{2\sqrt{\pi}(\mu\lambda_f)^{1/4}}\frac{e^{-(\sqrt{\mu}-\sqrt{\lambda_f})^2u}}{u^{3/2}},&\label{E: LowerUpperArg}
\end{eqnarray}
also for $x$ large enough. Denoting $\gamma=(\sqrt{\mu}-\sqrt{\lambda_f})^2,$ observe that
\begin{equation}
\label{E: AUB}
\left(\frac 1\gamma\frac 1{t^{3/2}}-\frac{4k^2-1}{2\gamma^2}\frac 1{t^{5/2}}\right)e^{-\gamma t}
\leq \int_t^{+\infty}\frac{e^{-\gamma u}}{u^{3/2}}\ du\leq \frac 1{\gamma}\frac{e^{-\gamma t}}{t^{3/2}}.
\end{equation}
Thus
\[
\int_t^{+\infty}\frac{e^{-\gamma u}}{u^{3/2}}\ du\sim \frac 1{\gamma}\frac{e^{-\gamma t}}{t^{3/2}},
\]
so we have a sharp asymptotic estimate for the integral of the upper bound in \eqref{E: LowerUpperArg}.

For the integral of the lower bound, denoting $\alpha=(4k^2-1)/(16\sqrt{\mu\lambda f})$, just observe that
\begin{equation}
\label{E: ALB}
\int_t^{+\infty}(1-\frac{\alpha}{u})\frac{e^{-\gamma u}}{u^{3/2}}\ du\geq
\int_t^{+\infty}\frac{e^{-\gamma u}}{u^{3/2}}\ du-\frac{\alpha}{\gamma}\frac{e^{-\gamma t}}{t^{5/2}}
\end{equation}
to see that we also have a sharp asymptotic estimate for the integral of the lower bound in \eqref{E: LowerUpperArg}.
Besides, both asymptotic estimates agree.

Plugging~(\ref{E: ALB}) and~(\ref{E: AUB}) into~(\ref{E: LowerUpperArg}) and then into~(\ref{T: ClosedForm})
we finally concluded that for $t$ large enough
\[
\bbP(\tau_f>t)\sim \left(\frac{\mu}{\lambda_f}\right)^{k/2}\frac 1{2\sqrt{\pi}(\mu\lambda_f)^{1/4}}\frac k{(\sqrt{\mu}-\sqrt{\lambda f})^2}
\frac{e^{-(\sqrt{\mu}-\sqrt{\lambda f})^2 t}}{t^{3/2}}.
\]
\end{proof}

\begin{proof}[Proof of Corollary~\ref{C: UpperBounds} (b)]
This immediate from the preceding computations and Remark \ref{R : transient}.
\end{proof}

\begin{proof}[Proof of Corollary~\ref{C: UpperBounds} (c)]
 When $\lambda_f=\mu$ \eqref{E: Survival} reads
\[
\bbP(\tau_f>t)=\int_t^{+\infty}e^{-2\mu u}\frac{k}{u}I_k(2\mu u)\ du
\]
using in turn inequalities \eqref{E: Bessel1} leads directly to the result.
\end{proof}

\end{document}